\documentclass[11pt,reqno]{amsart}

\usepackage{amsmath,enumerate,amsfonts,amsthm,amssymb,latexsym,verbatim,amscd,mathrsfs,pb-diagram, graphics,hyperref}
\usepackage[usenames,dvipsnames]{color}

\usepackage{geometry}
\theoremstyle{plain}
\newtheorem{thm}{Theorem}[section]
\newtheorem{lem}[thm]{Lemma}
\newtheorem{prop}[thm]{Proposition}
\theoremstyle{definition}
\newtheorem{ex}[thm]{Example}

\numberwithin{equation}{subsection}

\newcommand{\aut}{\operatorname{Aut}}
\newcommand{\der}{\operatorname{Der}}
\newcommand{\id}{\operatorname{id}}
\newcommand{\tam}{\operatorname{Tame}}

\title{On the tame isotropy group of a derivation}
\author{Angelo Bianchi, Adriana Freitas and Marcelo Veloso}

\begin{document}
	
	\begin{abstract}
		\noindent We introduce the \textit{tame isotropy group of a derivation of a polynomial ring}. We study this group for certain triangular derivations up to three variables, for simple derivations in two variables, and for simple Shamsuddin derivations in any polynomial ring.
	\end{abstract}
	\maketitle
	
	
	
	\section*{Introduction}
	
	Let $\mathcal R:=\mathbb K[X_1,\dots,X_n]$ be the polynomial ring on an algebraically closed field $\mathbb K$ of characteristic zero and let  $\aut(\mathcal R)$ be the group of $\mathbb{K}$-automorphisms of $\mathcal R$.
	
	Recall that a $\mathbb K$-derivation $D: \mathcal R \to \mathcal R$ is a $\mathbb K$-linear endomorphism of $\mathcal R$ satisfying the Leibniz rule $D(fg) = D(f)g+fD(g)$ for all $f, g \in \mathcal R$. We denote the set of $\mathbb K$-derivations of $\mathcal R$ by $\der(\mathcal R)$.
	In addition, recall that an ideal $I\subseteq \mathcal R$ is called \textit{$D$-stable} if $D(I) \subseteq I$. If the only $D$-stable ideals are the trivial ones, we say that $D$ is a \textit{simple derivation}.
	
	The definition of a simple derivation makes sense in any ring. Simple derivations are connected to other interesting topics in noncommutative algebra, commutative algebra, and algebraic geometry,  as very well mentioned in the introduction of \cite{BL}. However, examples of simple derivations are quite limited even in polynomial rings with no general recipe for a construction.
	
	The main objective of this work is to consider simple derivations in polynomial rings from the point of view of their iteration with the commuting automorphisms of the polynomial rings, which we will describe below.
	
	Let $D\in \der(\mathcal R)$. We denote by $\aut_D(\mathcal R)$ the subgroup of $\aut(\mathcal R)$ consisting of automorphisms that commute to $D$, i.e.,
	$$\aut_D(\mathcal R)=\{\rho \in \aut(R) \mid \rho D= D\rho \}=\{\rho \in \aut(R) \mid \rho D\rho^{-1}= D \}.$$
	In other words, $\aut_D(\mathcal R)$  is the isotropy of $D$ with respect to the natural action of $\aut(\mathcal R)$ on $\der(\mathcal R)$ by conjugation. Observe that simplicity is preserved by this action. We refer to $\aut_D(\mathcal R)$ as \textit{isotropy subgroup of $\mathcal R$ with respect to $D$}.
	
	From our point of view, the main distinguished works on general algebraic-geometric aspects of simplicity, we chronologically highlight the works of A. Seidenberg \cite{Seidenberg} and R. Block \cite{Block1968}, which outlined the approaches that followed since then. In subsequent decades, substantial progress was achieved by R. Hart \cite{Hart}, Brumatti, Lequain, and Levcovitz \cite{Brumatti2003}, R. Baltazar and I. Pan in \cite{B1}, and very recently P. Montagard, I. Pan and A. Rittatore in \cite{MPR}.
	
	An influential line of simple derivation examples comes from the class of Shamsuddin derivations in polynomial rings \cite{Coutinho2007,Leq,Shamsuddin1,Shamsuddin2} and from A. Nowicki \cite{Nowicki2008}. These works still guide the search for more general structural results, and the construction of simple derivations continues to be an active area (see \cite{Gavran2009,Kour2014,Maciejewski2001,MMS2,SP,Yan2019}.
	
	Recently, the structure of the isotropy subgroup of certain derivations of polynomial rings and other $\mathbb K$-algebras has been explored. Among the types of derivations most explored in this context, the following stand out: simple derivations (cf. \cite{B,BL,DY,MendesPan,MMS,MPR,RP,yan,yan2}) and locally nilpotent derivations (cf. \cite{BalVel,FW,MPR,P}) .
	
	The isotropy tame group can be used to establish results on simple derivations, but the interaction between simplicity and isotropy is subtle. Simplicity can force $\operatorname{Aut}_D(\mathcal R)$ to be trivial, but it is also possible to have an infinite isotropy group. In the following, we trace a few papers aligned to this interest:
	
	\begin{itemize}
		\item R. Baltazar in \cite{B} proved that a simple Shamsuddin derivation of $\mathbb K[X_1,X_2]$ has trivial isotropy group.
		\item L. Mendes and I. Pan in \cite{MendesPan} proved the converse of Baltazar's result and extended Baltazar's result for any simple derivation of $\mathbb K[X_1,X_2]$.
		\item L. Bertoncello and D. Levcovitz in \cite{BL} generalized Baltazar's result to a simple Shamsuddin derivation of $\mathbb K[X_1,\dots,X_n]$ with $n\ge 2$ with a proof strongly based on Lequain’s characterization of simple Shamsuddin derivations. Moreover, they conjectured that the reciprocal is also true, which was proved by D. Yan as we point out in the next item.   In this paper, the authors also conjectured that simplicity of a derivation is equivalent to the derivation having a trivial isotropy group, but this conjecture is false due to  examples of simple derivations with infinite isotropy group provided by D. Yan \cite[Theorems 2.2 and 2.4]{yan2}.
		
		\item D. Yan in \cite{yan} extended Bertoncello and Levcovitz's result by proving that $D$ is a simple Shamsuddin derivation of $\mathbb K[X_1,\dots,X_n]$, with $n\ge 2$, if, and only if, it has trivial isotropy group. Further, in \cite{yan2} Yan conjectured that the isotropy subgroup of a simple derivation of a polynomial ring is conjugate to a subgroup of translations.
	\end{itemize}
	
	The main object of this paper involves a \textit{refinement} of the definition of the isotropy subgroups: we take advantage of the distinguished automorphisms called \textit{elementary automorphisms}.  Recall that a $\mathbb K$-automorphism $\rho$ of $\mathcal R$ can be represented by $(g_1,\dots,g_n)$ where $g_i\in \mathbb K[X_1,\dots,X_n]$  and $\rho (x_i) = g_i$ for all $i\in\{1,\dots,n\}$. Under this notation, an \textit{elementary automorphism} of $\mathcal R$ is an automorphism of the form
	$$(X_1,\dots, X_n) \longmapsto (X_1,\dots, aX_i+f_i, \dots, X_n)$$
	where $i\in\{1,\dots,n\}$, $a\in \mathbb K\setminus\{0\}$ and $f_i \in \mathbb{K}[X_1,\dots, X_n]$ is independent of $X_i$, a \linebreak \textit{$j$-translation} of $\mathcal R$ is an elementary automorphism of the form
	$$(X_1,\dots,X_j,\dots, X_n) \longmapsto (X_1,\dots, X_j+c, \dots, X_n)$$
	where $j\in\{1,\dots,n\}$ and  $c\in \mathbb K\setminus\{0\}$, and a \textit{translation} of $\mathcal R$ is a composition of \linebreak $j$-translations. We denote the set of elementary automorphisms of $\mathcal R$ by $EL(\mathcal R)$. Recall that an automorphism of $\mathcal R$ is called \textit{tame} if it can be expressed as a composition of elementary automorphisms, and non-tame automorphisms are called \textit{wild}.
	
	It is straightforward to verify that any automorphism of $\mathbb{K}[X_1]$ is a tame automorphism. The same result remains valid in $\mathbb K[X_1,X_2]$ as a famous nontrivial result (cf. \cite{nagata73, Re1968,van2012book}). However, in $\mathbb{K}[X_1,X_2,X_3]$ the result fails as shown in \cite{shestakov042}.
	
	Given $D\in \der (\mathcal R)$, we define the \textit{tame isotropy group of $\mathcal R$ } as the subgroup generated by all elementary automorphisms that commute to $D$ and denote it by $\tam_D(R)$, i.e.
	$$\tam_D(\mathcal R) = \left\langle  \varphi \in \operatorname{EL}(\mathcal R) \,\middle|\, \varphi\in \aut_D(\mathcal R) \right\rangle.$$
	We reserve the symbol $\langle S \rangle$ to indicate the subgroup generated by finite compositions of elements in $S$.
	
	We see the first explicit mention of translations (as elements) in the isotropy group in \cite{yan2}, through examples and a conjecture that simplicity implies the  isotropy group acting by translations on $\mathcal R$. After that, in \cite{MPR}  the authors explore more insight on the structure of the isotropy group when it has a nontrivial translation. All of this prompted us to define this group as an object of interest in its own right.
	
	In order to illustrate and exemplify how to find the generators and structure of the tame isotropy group, we compute it for certain triangular derivations of $\mathbb K[X_1, \dots, X_n]$ for $n \leq 3$ (Section \ref{triangular}). As technical results in $\mathbb K[X_1, \dots, X_n]$, $n\ge 2$, we show that the tame isotropy group of certain simple derivations is generated only by translations (Proposition \ref{gen_translations}) and we present a class of simple derivations (Proposition \ref{no_translation}) whose isotropy group does not contain nontrivial translations.
	
	It is obvious that $$\aut_D(\mathcal R)=\{\id\} \implies \tam_D(\mathcal R)=\{\id\}.$$ However, it is still unknown when the reciprocal remains valid. One of the advantages of considering $\tam_D(\mathcal R)$ instead of $\aut_D(\mathcal R)$ is their explicit expression and the algorithmic way of testing commutation. Moreover, the translations are evidently elementary automorphisms.  In this direction, we show that every simple derivation of $\mathbb K[X_1,X_2]$ has a trivial tame isotropy group (Theorem \ref{generaln=2}) and we show that for a Shamsuddin derivation of $\mathbb K[X,Y_1,\dots,Y_n]$, $n\ge 1$, to be simple, it is equivalent to having a trivial tame isotropy group (Theorem \ref{shamsuddin}). Our proofs do not require advanced results on general automorphisms as those used in \cite{B,BL,MendesPan,MPR,yan}, respectively.
	
	\tableofcontents
	
	\section{Technical results on polynomials}
	
	The following two technical lemmas will be useful in some parts of this paper. Notice that these lemmas treat simple claims on polynomials and probably were already considered elsewhere.  We include a proof and comments on the proof, since we have not found a suitable reference.
	
	\begin{lem}
		\label{fx=fx+b}
		If  $f(X) \in \mathbb{K}[X]$ has degree $n\ge 1$ and $f(X)=f(X+c)$ for some $c \in \mathbb{K}$, then $c=0$.
	\end{lem}
	\proof
	Write
	$$f(X)=a_nX^n+a_{n-1}X^{n-1}+\cdots+a_0,$$ with $a_n\neq0$, $n\ge1$. By hypothesis, all coefficients of the polynomial
	$F(X):=f(X+c)-f(X)$ vanish. So, expanding $(X+c)^n$ by the binomial formula, the coefficient of $X^{n-1}$ in $(X+c)^n-X^n$ is $n c$. Thus, the coefficient of $X^{n-1}$ in $F(X)$ is equal to $a_n nc$. Since $a_n\neq0$ and $n\neq0$, we conclude that $c=0$. \endproof
	
	\begin{lem}  \label{xaxzaz}
		Let $g(X)\in \mathbb K[X]$ be a polynomial of degree $r\geq 2$ such that the coefficient of the term of degree $r-1$ is zero. If $\lambda,\beta \in \mathbb K$ and $\lambda\neq0$, then $g(X)=g(\lambda X+\beta)$ if, and only if, $\beta=0$, $\lambda$ is a $s$-th root of unity and $g(X)=h(X^s)$ for some $h(X) \in \mathbb K[X]$ and $s\in\mathbb N$.
	\end{lem}
	
	\proof
	The argument proceeds analogously to that of  \cite[Lemma 18]{BiaVel} using Taylor's expansion.
	\endproof
	
	\begin{lem} \label{divide}
		Let $R$ be a polynomial ring on $\mathbb K$ and $f(X) \in R[X]$ be a polynomial in the variable $X$ on $R$. Then, for all $a,b \in R[X]$, the difference $f(a) - f(b)$ is divisible by $a-b$ in $R[X]$.
	\end{lem}
	
	\begin{proof}
		Suppose  $f(X)=\sum_{i=0}^n c_i\,X^i$, $c_i \in R$. Then, $f(a)-f(b)=\sum_{i=0}^n c_i \big(a^i-b^i\big)$.
		For each $i\ge1$ the classical factorization 
		$$a^i-b^i=(a-b)\big(a^{\,i-1}+a^{\,i-2}b+\cdots+ab^{\,i-2}+b^{\,i-1}\big)$$
		holds in $R[X]$. So, each $a^i-b^i$ is divisible by $a-b$.
		Therefore, every summand on the right is divisible by $a-b$, and so is the whole sum, which proves the claim.
	\end{proof}
	
	\section{Triangular derivations}
	
	\label{triangular}
	
	Recall that a derivation of $\mathbb{K}[X_1,\dots, X_n]$ is said \textit{triangular} if
	$$D(X_1)\in \mathbb{K},\mbox{   and   } D(X_i)\in \mathbb{K}[X_1,\dots, X_{i-1}],$$
	for each $i=1,\dots,n$ in some coordinate system.
	
	\subsection{One variable case}
	
	Given that every automorphism in the polynomial ring $\mathbb K[X]$ is an elementary automorphism, it is immediate that
	$$
	\tam_D(\mathbb K[X])=\aut_D(\mathbb K[X]),
	$$
	for any $D\in \der(\mathbb K[X])$. In particular, this holds true for triangular derivations. Moreover, the explicit description of $\tam_D(\mathbb K[X])$ is the following:
	
	\begin{prop}
		Let $D=a\frac{\partial}{\partial {X}}$ be a non-zero triangular derivation  of $\mathbb K [X]$. Then
		$$
		\tam_D(\mathbb{K}[X])=\langle X+b \mid b \in \mathbb{K}\rangle.
		$$
	\end{prop}
	\proof
	Recall that if $\rho \in Aut(\mathbb{K}[X])$, then $\rho$ is a tame automorphism, i.e. \linebreak $\rho(X)=cX+b$, where $b,c \in \mathbb{K}$ and $c\neq 0$. Then,
	$$ca=D(b+cX)=D(\rho(X))=\rho(D(X))=\rho(a)=a,	$$
	from where we get $c=1$ since $a \ne 0$.
	\endproof
	
	\subsection{Two variables case}
	
	Let
	$$
	D=c\frac{\partial}{\partial{X}}+f(X)\frac{\partial}{\partial{Y}}
	$$
	of $\mathbb{K}[X,Y]$, treating the cases separately $c\ne 0$ and $c=0$. The first case is simpler and will be treated directly, whereas the second requires a restriction on the polynomial $f(X)$.
	
	\begin{thm} \label{tigtc} Let $D=c\frac{\partial}{\partial{X}}+f(X)\frac{\partial}{\partial{Y}}$ be the non-zero triangular derivation of $\mathbb{K}[X,Y]$, where $0\neq c\in\mathbb K$ and $f(X) \in \mathbb{K}[X]$ has degree $n\geq 1$.  Then,
		$$
		\tam_D(\mathbb{K}[X,Y])=\left\langle (X, \beta Y+h(X)) \mid h'(X)=\dfrac{1-\beta}{c}f(X)  \text{ and } \beta \in \mathbb{K}\right\rangle.
		$$
	\end{thm}
	
	\proof We proceed explicitly determining which elementary automorphisms commute to $D$:
	\begin{enumerate}
		
		\item  Let $\theta \in Aut(\mathbb{K}[X,Y])$ be the automorphism defined by
		$$
		\theta(X)=\alpha X+ g(Y)  \mbox{    and   } 	\theta(Y)=Y
		$$
		where $\alpha \in \mathbb{K}$ and $g \in \mathbb{K}[Y]$. Suppose that $\theta D= D \theta$. Then,
		$$
		c=\theta(c)=\theta(D(X))=D(\theta(X))=D(\alpha X+ g(Y))=\alpha c + g'(Y)f(X).
		$$
		This implies $g'(Y)=0$  and  $\alpha = 1$, since $c\ne 0$ and $\deg_X f=n\ge 1$. Thus,
		$$
		\theta(X)=X+ \gamma  \mbox{    and   } 	\theta(Y)=Y,
		$$
		where $\gamma \in \mathbb{K}$. Also,
		$$
		f(X)=D(Y)=D(\theta(Y))=\theta(D(Y))=\theta(f(X))=f(\theta(X))=f(X+ \gamma).
		$$
		Hence, $f(X)=f(X+ \gamma)$. Lemma \ref{fx=fx+b} implies that   $\gamma=0$. Hence,  $\theta =id$.
		\item
		
		Let $\rho \in Aut(\mathbb{K}[X,Y])$ be the automorphism
		$$
		\rho(X)=X,\,\, \rho(Y)=\beta Y +h(X)
		$$
		where $\beta \in \mathbb{K}$ and $h \in \mathbb{K}[X]$. Suppose that $\rho D= D \rho$. Then,
		$$
		D(\rho(X))=D(X)=c=\rho(c)=\rho(D(X)).
		$$
		Also,
		$$
		f(X)=\rho(f(X))=\rho(D(Y))=D(\rho(Y))=D(\beta Y +h(X))=ch'(X)+\beta f(X).
		$$
		Hence,
		$$
		\rho(X)=X,\,\, \rho(Y)=Y +h(X),
		$$
		where  $h(X)\in\mathbb K[X]$ satisfies $h'(X)=\frac{1-\beta}{c}f(X)$ ($c\ne0$) and $\beta \in \mathbb K$.
	\end{enumerate}
	
	\endproof
	
	\
	
	Given any polynomial $f\in\mathbb K[X]$ of degree $n\geq 2$, we can simplify the polynomial $f(X)$ by eliminating the coefficient of $X^{n-1}$. To do this, if $f(X)=\displaystyle\sum_{i=0}^{n} a_iX^i$ we replace $X$ by $X - \frac{a_{n-1}}{n a_n}$ and the degree coefficient $n-1$ in $f(X)$ is zero.
	
	\begin{thm} \label{tigt0}  Let $D=f(X)\frac{\partial}{\partial Y}$ be the non-zero triangular derivation of $\mathbb{K}[X,Y]$, where   $f(X) \in \mathbb{K}[X]$ has degree $n\geq 2$ and its coefficient of the term of degree $n-1$ is zero.   If $f(X)=h(X^s)$ for some $h\in\mathbb K[X]$ and $s\in \mathbb N$, then
		$$\tam_D(\mathbb{K}[X,Y])= \langle (X, Y+r(X)) \text{ and } (\lambda X, Y) \mid r(X) \in \mathbb{K}[X] \text{ and } \lambda \text{ a } \text{ $s$-th root unity}\rangle.$$
	\end{thm}
	\proof[Sketch of proof]
	The verification is analogous to the previous result. It suffices to directly check which elementary automorphisms commute to the derivation \(D\),
	and then apply  Lemma \ref{xaxzaz}.
	\endproof
	
	\subsection{Three variables case}
	
	In a manner similar to the previous section, we calculate the tame isotropy group of certain triangular derivations in three variables under mild assumptions.
	
	\begin{thm}
		\label{Taut}
		Let $D=f(X)\frac{\partial}{\partial{Y}}+g(X,Y))\frac{\partial}{\partial{Z}}$ be a non-zero triangular derivation of $\mathbb{K}[X,Y,Z]$, where   $f(X) \in \mathbb{K}[X]$ and  $g(X,Y) \in \mathbb{K}[X,Y]$ have no common irreducible factors in the ring $\mathbb{K}[X,Y]$.  If $f(X)=h(X^s)$ for some $h\in\mathbb K[X]$ and $s\in \mathbb N$, and $g(\lambda X,Y)=g(X,Y)$ for a $s$-th root unity $\lambda\in\mathbb K$, then
		$$\tam_D(\mathbb{K}[X,Y])= $$ $$\langle (X,Y, Z +r(X)) \text{ and } (X, \lambda Y, Z)  \mid r(X) \in \mathbb{K}[X] \text{ and } \lambda \text{ a } \text{ $s$-th root unity}\rangle.$$
	\end{thm}
	
	\proof Again, we proceed by explicitly determining which elementary automorphisms commute to $D$:
	\begin{enumerate}
		\item
		Let $\rho \in Aut(\mathbb{K}[X,Y, Z])$ be the elementary automorphism
		$$
		\rho(X)=X,\,\, \rho(Y)=Y \mbox{    and   } 	\rho(Z)=bZ+h(X,Y)
		$$
		where $b \in \mathbb{K}$ and $h \in \mathbb{K}[X,Y]$. Suppose that $\rho D= D \rho$. Then, $$\rho(D(Y))=f(X)=D(\rho(Y)),$$ $$\rho(D(X))=0=D(\rho(X)),$$
		and
		$$
		\begin{aligned}
			g(X,Y)= \rho(g(X,Y))&=\rho(D(Z))\\
			&=D(\rho(Z))\\
			&=D(bZ+h(X,Y))\\
			&=bD(Z)+h_{X}(X,Y)D(X)+h_{Y}(X,Y)D(Y)\\
			&=bg(X,Y)+0+h_{Y}(X,Y)f(X)
		\end{aligned}
		$$
		which only requires
		$$
		(1-b)g(X,Y)=h_{Y}(X,Y)f(X)
		$$
		If $1-b=0$, then $h_{Y}(X,Y)=0$. Thus,  $h(X,Y)$ does not involve $Y$. In this case \linebreak $\rho=(X,Y, Z +r(X))$.
		If $1-b \neq 0$, then $f(X)$ divides $g(X,Y)$, contradicting the main hypothesis on $f$ and $g$.
		
		\item
		Let $\theta \in Aut(\mathbb{K}[X,Y, Z])$ be the elementary automorphism
		$$
		\theta(X)=X,\,\, \theta(Y)=bY+h(X,Z)  \mbox{    and   } 	\theta(Z)=Z
		$$
		where $b \in \mathbb{K}$ and $h \in \mathbb{K}[X,Z]$. Suppose that $\theta D= D \theta$. First, notice that
		$$
		\theta(D(Y))=\theta(f(X))=f(\theta(X))=f(X),$$
		and
		$$
		\begin{aligned}
			D(\theta(Y))=&D(bY+h(X,Z))\\
			=&bD(Y)+h_{X}(X,Z)D(X)+h_{Z}D(Z)\\
			=&bf(X)+0+h_{Z}(X,Z)g(X,Y)
		\end{aligned}
		$$
		implying
		$$
		(1-b)f(X)=h_{Z}(X,Z)g(X,Y).
		$$
		Proceeding as in the conclusion of Item (1), we obtain
		$$
		\theta=(X,Y+s(X), Z).
		$$
		Moreover, since $D(\theta(Z))=\theta(D(Z))$, 
		$$
		g(X,Y)=g(X,Y+s(X))
		$$
		implying $s(X)=0$. In fact, given $\alpha \in \mathbb{K}$, by setting  $	r(Y)=g(\alpha, Y)$ 	we get  $r(Y)=r(Y+s(\alpha))$ and, by Lemma $\ref{fx=fx+b}$, we conclude that $s(\alpha)=0$, which means $s=0$ since $\alpha \in \mathbb K$ is arbitrary, concluding that $\theta=\id$.
		
		\item
		Let $\sigma \in Aut(\mathbb{K}[X,Y, Z])$ be the elementary automorphism
		$$
		\sigma(X)=bX+h(Y,Z),\,\, \sigma(Y)=Y \mbox{    and   } 	\sigma(Z)=Z
		$$
		where $b \in \mathbb{K}$ and $h \in \mathbb{K}[Y, Z]$. Suppose that $\sigma D= D \sigma $. Then,
		$$\begin{aligned}
			0=\sigma(0)=\sigma(D(X))=&D(\sigma(X))\\
			=&D(bX+h(Y,Z))\\
			=&bD(X)+h_{Y}D(Y)+h_{Z}D(Z)\\
			=&0+ h_{Y}f(X)+h_{Z}g(X,Y)\\
		\end{aligned}
		$$
		implying
		$$
		h_{Y}(Y,Z)f(X)+h_{Z}(Y,Z)g(X,Y)=0.
		$$
		If $h(Y,Z)\in \mathbb K$  we immediately have  $\sigma(D(X))=D(\sigma(X))$. If $h(Y,Z)$ is non-constant and
		$$
		h(Y,Z)=a_0(Y)+a_1(Y)Z+\dots+a_n(Y)Z^n,
		$$
		where $a_0(Y), \dots,  a_n(Y) \in \mathbb{K}[Y]$ and $a_n(Y) \neq 0$.
		Then,
		$$
		h_{Y}(Y,Z)=a'_0(Y)+a'_1(Y)Z+\dots+a'_{n-1}(Y)Z^{n-1}+a'_n(Y)Z^n
		$$
		and
		$$
		h_{Z}(Y,Z)=a_1(Y)+2a_2(Y)Z+\dots+(n-1)a_{n-1}(Y)Z^{n-2}+na_n(Y)Z^{n-1},
		$$
		where $a'_i=\frac{\partial(a_i)}{\partial Y}$. So,
		$$
		f(X)a'_n(Y)=0
		$$
		and
		$$
		f(X)a'_{n-1}(Y)+g(X,Y)na_n(Y)=0.
		$$
		Thus, $a'_n(Y)=0$. So, $0\neq a_n(Y)=c \in \mathbb{K}$ and
		$$
		f(X)a'_{n-1}(Y)+cng(X,Y)=0.
		$$
		In particular, $a'_{n-1}(Y)=0$ because $f(X)$ does not divide $g(X,Y)$. Therefore,  $cng(X,Y)=0$, implying  $n=0$, since $c= a_n(Y)\neq 0$. Hence, $h(Y,Z)=a_0(Y)$ and $a'_0(Y)=0$, concluding that $h(Y,Z) \in \mathbb{K}$ and
		$$
		\sigma=(bX+c,Y, Z).
		$$
		Now, once we also have $$f(X)=D(\sigma(Y))=\sigma(D(Y))=f(bX+c),$$ by Lemma $\ref{xaxzaz}$ we conclude that $c=0$, $b$ is a $s$-th root of unity, and $f(X)=h(X^s)$
		with  $h(X) \in \mathbb{K}[X]$.
		
		Finally,
		$$g(bX,Y)=\sigma(g(X,Z))=\sigma(D(Z))=D(\sigma(Z))=D(Z)=g(X,Y) $$
		implying 
		$$
		g(X,Y)=g(bX,Y).
		$$
	\end{enumerate}
	\endproof
	
	\section{On certain simple derivations on a polynomial ring in several variables}
	
	We call a $j$-translation any translation $\sigma$ in $\aut(\mathbb K[X_1,\dots,X_n])$ such that $\sigma(X_i)=X_i$ for all $i\in\{1,\dots,n\}\setminus\{j\}$ and $\sigma(X_j)=X_j+c$ for some $c\in \mathbb K\setminus\{0\}$.
	
	\begin{prop}\label{gen_translations}
		Let $D =\sum_{i=1}^n f_i \frac{\partial}{\partial X_i}$, where $n\ge 2$, $f_i \in \mathbb K[X_1,\dots,X_n]$ for all \linebreak $i\in \{ 1,\dots, n\}$.  If $D$ is a simple derivation, then  $\tam_D(\mathbb K[X_1,\dots,X_n])$ is generated only by translations. 
	\end{prop}
	
	\proof Suppose $D$ is simple and $\sigma$ is an elementary automorphism that commutes to $D$.
	
	Let $j\in \{1,\dots,n\}$ be such that $\sigma(X_i)= X_i$ for all $i\in \{1,\dots,n\}\setminus\{ j\}$ and \linebreak $\sigma(X_j) = \alpha X_j + h$, where $\alpha \in \mathbb K$, $h \in \mathbb K[X_1,\dots,X_n]$, $\alpha\ne 0,$ and $h$ does not involve $X_j$.
	
	Since $\sigma D = D \sigma$, by applying this condition to $X_1,\dots,X_n$, we get:
	\begin{equation} \label{1.1}
		f_i=D(X_i)=D(\sigma(X_i))=\sigma(D (X_i))=\sigma(f_i)  \text{ for all } i\in \{1,\dots,n\}\setminus\{ j\}
	\end{equation}
	and
	\begin{equation} \label{1.2}
		\alpha f_j+D(h)=D(\alpha X_j+h)=D(\sigma(X_j))=\sigma (D (X_j))=\sigma(f_j).
	\end{equation}
	Setting $u:=\sigma(X_j)-X_j=(\alpha-1)X_j+h$ we get
	$$D(u)=(\alpha-1) f_j + D(h) = \sigma(f_j)-f_j,$$ which is divisible by $\sigma(X_j)-X_j=u$ by the Lemma \ref{divide}. Hence, $D(u)\in \langle u \rangle$, i.e. the principal ideal $\langle u \rangle\subseteq \mathbb K[X,Y]$ is $D$-stable.  By the simplicity of $D$, either $\langle u \rangle=0$ or $ \langle u \rangle=\mathbb K[X,Y]$, which means $u=(\alpha-1)X_j+h\in\mathbb K$. Finally, by degree comparison and the hypothesis that $h$ does not involve $X_j$, we conclude that $\alpha=1$ and $h\in\mathbb K$. Hence, we conclude that $\sigma$ is a translation. \endproof
	
	\
	
	\begin{prop}\label{no_translation}
		Let $D = \sum_{i=1}^n f_i \frac{\partial}{\partial X_i}$, where $n\ge 2$, $f_i \in \mathbb K[X_1,\dots,X_n]$ for all \linebreak $i\in \{ 1,\dots, n\}$, be a simple derivation. For each $j\in\{1,\dots, n\}$, if $\deg_{X_j} f_i\ge 1$ for some $i\in\{1,\dots,n\}$, then there is no $j$-translation in $\aut_D(\mathbb K[X_1,\dots,X_n])$.
	\end{prop}
	
	\proof  Let $j\in\{1,\dots,n\}$ and $\sigma$ a $j$-translation commuting to $D$ given by
	$$\sigma(X_j) = X_j + c, \quad \text{ where } c \in \mathbb K\setminus\{0\},$$
	and
	$$ \sigma(X_i)= X_i  \text{ for all } i\in \{1,\dots,n\}\setminus\{ j\}.$$
	
	Since $\sigma D = D \sigma$, by applying this condition to $X_1,\dots,X_n$, we have
	\[
	f_i=D(X_i)=D(\sigma(X_i))=\sigma(D (X_i))=\sigma(f_i)  \text{ for all } i\in \{1,\dots,n\}\setminus\{ j\}
	\]
	\[
	f_j=D(X_j+c)=D(\sigma(X_j))=\sigma (D (X_j))=\sigma(f_j).
	\]
	So, in a unified way, we have $\sigma(f_i)=f_i$ for all  $i\in \{1,\dots,n\}$. Thus, if $\deg_{X_j} f_i\ge 1$ for some $i\in\{1,\dots,n\}$, by Lemma \ref{fx=fx+b} we conclude that $c=0$, that means $\sigma=\id$, contradicting the assumption on $\sigma$.  \endproof

	\section{Simple derivations on a polynomial ring in two variables}
	
	In this section, we focus on simple derivations of the polynomial rings in two variables $\mathbb K[X,Y]$. We begin with some properties which are immediate consequences of simplicity:

	\begin{lem} \label{simplification}
		Let $D = f \frac{\partial}{\partial X} + g \frac{\partial}{\partial Y} \in \der_\mathbb K (\mathbb K[X,Y])$ be a simple derivation.
		
		\begin{enumerate}
			\item \label{spart1} If $f\in \mathbb K$, then $g\notin \mathbb K$. Analogously, if $g\in \mathbb K$, then $f\notin \mathbb K$.
			
			\item \label{spart2} If $g\in\mathbb K[Y]$, then $g\in\mathbb K$. 
			Similarly, if $f\in\mathbb K[X]$, then $f\in\mathbb K$.
			
			\item \label{spart3} If $g\in\mathbb K[X]\setminus\mathbb K$, then $f\notin \mathbb{K}$. Analogously, if $f\in\mathbb K[Y]\setminus\mathbb K$, then $g\notin \mathbb{K}$.
			
			\item \label{spart4} $\deg_X f + \deg_X g\ge 1$ and $\deg_Y f+\deg_Y g\ge 1$.
		\end{enumerate}
	\end{lem}
	\proof  \begin{enumerate}
		\item If $f=a$ and $g=b$ are constants, then  $H=bX-aY$ satisfies  $D(H)=0$. Thus, $\langle H \rangle$ is a $D$-stable ideal,  contradicting the simplicity of $D$. 
		
		\item If $g\in \mathbb K[Y]$, then $D(g)\in \langle g \rangle$, implying $\langle g\rangle$ is a $D$-stable ideal, a contradiction to the simplicity of $D$. The proof for $f\in \mathbb K[X]$ is analogous.
		
		\item Suppose $f=c\in \mathbb{K}^*$. Let $h(X)\in\mathbb K[X]\setminus\mathbb K$ such that $h'=g$ and set $H=\frac{1}{c}h-Y$. Then, $D(H)=c \frac{\partial H }{\partial X} + g \frac{\partial H}{\partial Y} = h'-g=0$. So, $ \langle H \rangle$ is a $D$-stable ideal, a contradiction to the simplicity of $D$. 
		
		\item If $f, g \notin \mathbb{K}$, item (2) implies $\deg_X g\ge 1$ and $\deg_Y f \ge 1$. If $f \notin \mathbb{K}$ and $g\in \mathbb{K}$, items (2) and (3) give $\deg_Y f\ge 1$ and $\deg_X f \ge 1$. Similarly, if $g \notin \mathbb{K}$ and $f\in \mathbb{K}$ we obtain $\deg_X g\ge 1$ and $\deg_Y g \ge 1$. The result follows immediately from these inequalities.
	\end{enumerate}
	\endproof

	The main result of this section is:
	
	\begin{thm}\label{generaln=2}
		Let $D = f \frac{\partial}{\partial X} + g \frac{\partial}{\partial Y} $ where $f,g \in \mathbb K[X,Y]$. If $D$ is a simple derivation, then $\tam_D(\mathbb K[X,Y])=\{ \id\}$.
	\end{thm}
	
	\proof  Let $\sigma,\tau$ be the elementary automorphism that commutes to $D$, defined by
	$$\sigma(X)= a X+h(Y)   \text{ and } \sigma(Y)=Y, \text{ where } a\in\mathbb K \text{ and } h\in \mathbb K[Y],$$
	and
	$$ \tau(Y)= bY+s(X)   \text{ and } \tau(X)=X, \text{ where } b\in\mathbb K \text{ and } s\in \mathbb K[X].$$
	
	Applying $D\sigma=\sigma D$ on $X$ and $Y$, respectively, we have
	$$af(X,Y)+g(X,Y)h'(Y)= D(\sigma(X))=\sigma D (X)= f(aX+h(Y),Y)$$
	and
	$$g(X,Y)=D(\sigma(Y))=\sigma D (Y)= g(aX+h(Y),Y),$$
	providing 
	\begin{equation} \label{cond1.1}
		\begin{cases}
			g(X,Y) = g(aX+h(Y),Y) \\
			af(X,Y)+g(X,Y)h'(Y) =  f(aX+h(Y),Y).
		\end{cases}
	\end{equation}
	
	Set $u:=\sigma(X)-X=(a-1)X+h(Y)\in\mathbb K[X,Y]$. Then,
	$$D(u)=(a-1)f(X,Y)+g(X,Y)h'(Y) = f(aX+h(Y),Y)-f(X,Y).$$
	Now, using the standard algebraic argument that $T-X$ divides $f(T,Y)-f(X,Y)$ in $\mathbb{K}[X,Y]$, we conclude that $f(aX+h(Y),Y)-f(X,Y)$ is divisible by $(aX+h(Y))-X=u$. Hence, $D(u)\in \langle u \rangle$ and the ideal $\langle u \rangle\subseteq\mathbb K[X,Y]$ is $D$-stable.
	By simplicity of $D$, the ideal $\langle u \rangle$ is either $\langle 0 \rangle$ or the whole ring. Therefore, the polynomial $u=(a-1)X+h(Y)$ is a constant polynomial, giving us $a-1=0$ and $h(Y)=c\in \mathbb K$. So, \eqref{cond1.1} becomes
	\begin{equation}  \label{4.2}
		\begin{cases}
			g(X,Y) = g(X+c,Y) \\
			f(X,Y) =  f(X+c,Y).
		\end{cases}
	\end{equation}
	
	On the other hand, we have a symmetric argument using $\tau$ that leads us to conclude
	\begin{equation}  \label{4.3}
		\begin{cases}
			g(X,Y) = g(X,Y+d) \\
			f(X,Y) =  f(X,Y+d).
		\end{cases}
	\end{equation}
	for some $d\in\mathbb K$.

	Finally, by Lemma \ref{simplification},  $\deg_X f + \deg_X g\ge 1$ and $\deg_Y f+\deg_Y g\ge 1$. By Lemma \ref{fx=fx+b} and equations \eqref{4.2} and \eqref{4.3}, we conclude that $c=d=0$ which gives $\sigma=\tau=\id$. \endproof

	Next example shows that the converse of Theorem \ref{generaln=2} is not valid as well as
	$$
	\aut_D(\mathbb{K}[X,Y]) \ne \tam_D(\mathbb{K}[X,Y]).
	$$
	
	\begin{ex}
		Let $D\in Der(\mathbb{K}[X,Y])$ defined by
		$$ D=-Y\dfrac{\partial}{\partial X}+X\dfrac{\partial}{\partial Y}.$$
		Then, $D$ is not simple, since $D(X^2+Y^2)=0$. Moreover, it is not difficult to see that $\tam_D(\mathbb{K}[X,Y])=\{ \id \}$. Indeed, if  $\sigma=(\alpha X+g(Y),Y)$ where $\alpha \in \mathbb K$ and $g(Y) \in \mathbb K[Y]$, commutes to $D$, then
		$$\sigma(D(Y)) = \sigma(X) = \alpha X+g(Y)  \text{ and } D(\sigma(Y)) = D(Y)=X$$
		implies $g(Y)=0$ and $\alpha =1$. Therefore, $\sigma=\id$. Similarly, we also conclude  that the elementary automorphism $\theta=(X, \beta Y+h(X))$, where $\beta \in \mathbb K$ and $h(X) \in \mathbb K[X]$, commutes to $D$ only if it is the trivial automorphism. However, the automorphism $\rho \in \aut(\mathbb{K}[X,Y])$ defined by
		$$\rho(X)=X-Y \mbox{ and }\rho(Y)=X+Y$$
		commutes to $D$, since
		$$D(\rho(Y))=D(X+Y)=-Y+X =  \rho(X) = \rho(D(Y)),$$
		$$D(\rho(X))=D(X-Y)=-Y-X =- \rho(Y) = \rho(D(X)),$$
		and it can be written (not uniquely) as a composition of elementary automorphisms, for instance
		\[\rho=(X-Y,Y)(2X,Y)(X,X+Y)
		\]
		but none of these elementary automorphisms is in $\tam_D(\mathbb{K}[X,Y])=\{ \id\}$.
		Hence,
		\[
		\aut_D(\mathbb{K}[X,Y]) \ne \tam_D(\mathbb{K}[X,Y]).   
		\]
		$\hfill\diamond$
	\end{ex}

	Next example illustrates us that a non simple and non locally nilpotent derivation can admit a nontrivial tame automorphism.
	
	\begin{ex}
		Let $D = bY\frac{\partial}{\partial Y}$, $b \in \mathbb{K}^*$, be a derivation of $\mathbb{K}[X,Y]$.
		Observe that $D$ is not simple, since $D(X) \subseteq \langle X \rangle$, and $D$ is not a locally nilpotent derivation, since $D^n(Y) = b^n Y$ for all $n \geq 1$. Moreover,
		$$
		\operatorname{Tame}_D(\mathbb{K}[X,Y])
		= \left\langle (\alpha X + \gamma,\, Y) \text{ and } (X,\, \beta Y)
		\;\middle|\;
		\alpha, \beta, \gamma \in \mathbb{K},\ \alpha\beta \neq 0
		\right\rangle.
		$$
		The verification is direct and we omit the details to avoid  repetition. $\hfill\diamond$ 
	\end{ex}

	\section{Shamsuddin derivations}
	
	In this section, we focus on the Shamsuddin derivations in $n+1$ variables, with $n\ge 1$. The notation for the polynomial ring follows the  standards of the literature on the subject as $\mathbb K[X, Y_1,\dots, Y_n]$.
	
	\

	We begin with a technical result:
	
	\begin{lem} \label{deg>0}
		Let
		$$D = \frac{\partial}{\partial X} + \sum_{i=1}^n (a_i(X)Y_i + b_i(X)) \frac{\partial}{\partial Y_i}$$
		be a \textit{Shamsuddin derivation} of $\mathbb K[X, Y_1, \dots, Y_n]$, where $a_i(X), b_i(X) \in \mathbb K[X]$ for all \linebreak $i\in\{1,\dots,n\}$. If $D$ is simple, then
		\begin{enumerate}
			\item \label{deg>0.1} $a_i\ne 0$;
			\item \label{deg>0.2} $\deg_X(a_i)\ge 1$ or $\deg_X(b_i)\ge 1$ for all $i\in\{1,\dots,n\}$.
		\end{enumerate}
	\end{lem}
	
	\proof \begin{enumerate}
		\item Suppose that $a_i=0$ for some $i\in \{1,\dots,n\}$. Let
		$B_i(X)\in\mathbb K[X]$ such that $B_i'(X)=b_i(X)$. Then, $I=\langle Y_i - B_i(X)\rangle$ is a nontrivial $D$-stable ideal, a contradiction with the simplicity of $D$.
		
		\item Suppose $a_i, b_i \in \mathbb K$ for some index $i\in\{1,\dots,n\}$ and set $f_i:= a_i Y_i + b_i$. Then   $D(f_i) = a_i D(Y_i) = a_i (a_i Y_i + b_i) = a_i f_i$. So,  $\langle f_i \rangle$ is a 
		nontrivial $D$-stable ideal, contradicting the simplicity of $D$. Therefore, $a_i \notin \mathbb K$ or $b_i \notin \mathbb K$.
		
	\end{enumerate}
	\endproof
	
	The main result of this section is:
	
	\begin{thm} \label{shamsuddin}
		Let
		$$D = \frac{\partial}{\partial X } + \sum_{i=1}^n (a_i(X)Y_i + b_i(X)) \frac{\partial}{\partial Y_i}$$
		be a \textit{Shamsuddin derivation} of $\mathbb K[X, Y_1, \dots, Y_n]$, where $a_i(X), b_i(X) \in \mathbb K[X]$ for all \linebreak $i\in\{1,\dots,n\}$. Then
		$$	D \text{ is simple} \quad \Longleftrightarrow \quad \tam_D(\mathbb K[X, Y_1, \dots, Y_n]) = \{ \id \}.$$
	\end{thm}
	
	\proof Suppose that $D$ is simple. By Proposition \ref{gen_translations}, we have $\tam_D(\mathbb K[X, Y_1, \dots, Y_n])$ generated only by translations. Let $\sigma$ be a translation that commutes to $D$. If $\sigma \ne \id$, it differs from the identity in $X$ or in $Y_i$ for some $i\in\{1,\dots,n\}$.
	
	First, let $i\in\{1,\dots, n\}$, and consider 
	$$\sigma(X)=X, \ \sigma(Y_j)=Y_j \text{ for all } j\in\{1,\dots,n\}\setminus \{ i\}, \  \text{ and } \sigma(Y_i) = Y_i + \alpha,$$ with $\alpha \in \mathbb K$.    Since $\sigma D = D \sigma$, by applying both sides to $Y_i$, we get
	$$\sigma (D(Y_i)) = a_i(X)\sigma(Y_i) + b_i(X) = a_i(X)(Y_i + \alpha) + b_i(X),$$
	$$D(\sigma(Y_i)) = D(Y_i + \alpha) = D(Y_i) = a_i(X)Y_i + b_i(X).$$
	Then,  $a_i(X) \alpha = 0$ and, by Lemma \ref{deg>0} \eqref{deg>0.1}, $\alpha = 0$.
	
	Similarly, if we have  
	\[\sigma(Y_i)=Y_i \text{ for all } i\in \{1,\dots,n\},\ \text{ and }
	\sigma(X) =   X + \alpha, \quad \text{ where } \alpha \in \mathbb K.
	\]
	Then,
	$$1 = D(X + \alpha)=D(\sigma(X))=\sigma( D(X) )= 1$$
	$$a_i(X)Y_i+b_i(X)= D(Y_i)=D(\sigma(Y_i))=\sigma( D(Y_i)) = a_i(X + \alpha)Y_i+b_i(X + \alpha),$$
	for all $i\in \{1,\dots,n\}$. Therefore, $\alpha =0$, since $\deg_X a_i\ge 1$ or $\deg_X b_i\ge 1$, according to Lemma \ref{deg>0}  \eqref{deg>0.2} again.
	
	From these two cases, we conclude that $\tam_D(\mathbb K[X, Y_1, \dots, Y_n])=\{ \id\}$.
	
	\
	
	Reciprocally, suppose, for contradiction, that $D$ is not simple. By the classical characterization of Shamsuddin derivations, there exists an index $i\in\{1,\dots,n\}$ and a polynomial $h\in\mathbb K[X]$ satisfying
	\begin{equation}\label{eq:ODE}
		h'(X) = a_i(X)h(X) + b_i(X).
	\end{equation}
	In this case, the principal ideal $\langle Y_i - h(X) \rangle$ is $D$-stable, since
	$$D(Y_i - h(X)) = a_i(X)Y_i + b_i(X) - h'(X)
	= a_i(X)\bigl(Y_i - h(X)\bigr)
	\in \langle Y_i - h(X) \rangle.$$
	
	We now construct a nontrivial elementary automorphism $\sigma$ commuting with $D$: define $\sigma$ on the generators by
	$$  \sigma(X) = X, \qquad
	\sigma(Y_i) = 2Y_i - h(X), \qquad
	\sigma(Y_j) = Y_j \quad (j\neq i).
	$$
	We claim $\sigma D = D \sigma$. Indeed, this is immediate for $X$ and $Y_j$ with $j\neq i$. For $Y_i$, we have
	$$ \sigma(D(Y_i))
	= \sigma\bigl(a_i(X)Y_i + b_i(X)\bigr)
	= a_i(X)\bigl(2Y_i - h(X)\bigr) + b_i(X), $$
	$$D(\sigma(Y_i))
	= D\bigl(2Y_i - h(X)\bigr)
	= 2\bigl(a_i(X)Y_i + b_i(X)\bigr) - h'(X).
	$$
	Using \eqref{eq:ODE}, we obtain
	$$ D(\sigma(Y_i))
	= a_i(X)\bigl(2Y_i - h(X)\bigr) + b_i(X)
	= \sigma(D(Y_i)).
	$$
	Thus, $\sigma D = D \sigma$ and $\sigma \ne \id$, contradicting $\tam_D(\mathbb K[X, Y_1, \dots, Y_n]) = \{\id\}$. Hence, $D$ must be simple.
	\endproof

\end{document}